\documentclass[12pt]{amsart}
\usepackage[margin=1.25in]{geometry} 
\usepackage{amsmath,amssymb,amsthm,enumerate,setspace}
\usepackage{mathrsfs}
\usepackage[mathscr]{euscript}
\usepackage{mathtools}
\usepackage{latexsym}
\usepackage{hyperref}
\usepackage{polynom}
\usepackage{enumitem}
\usepackage{xcolor}
\usepackage{bm}
\usepackage{relsize}

\newtheorem{definition}{Definition}[section]
\newtheorem{theorem}[definition]{Theorem}
\newtheorem{lemma}[definition]{Lemma}

\newtheorem{example}[definition]{Example}
\newtheorem{conjecture}[definition]{Conjecture}
\newtheorem{remark}[definition]{Remark}

\def\Z{\mathbb{Z}}
\def\Q{\mathbb{Q}}

\def\F{\mathbb{F}}
\def\l{\left}
\def\r{\right}

\def\Gal{\text{Gal}}

\title{Lower Bounds for $\text{GL}_2(\mathbb{F}_\ell)$ Number Fields}
\author{Vittoria Cristante}

\begin{document}
\maketitle

\begin{abstract}
Let $\mathcal{F}_n(X;G)$ denote the set of number fields of degree $n$ with absolute discriminant no larger than $X$ and Galois group $G$. This set is known to be finite for any finite permutation group $G$ and $X \geq 1$. In this paper, we give a lower bound for the cases $G=\text{GL}_2(\mathbb{F}_\ell), \; \text{PGL}_2(\mathbb{F}_\ell)$ for primes $\ell \geq 13$. We also provide a method to compute lower bounds for any permutation representations of these groups. 
\end{abstract}

\section{Introduction}

For a transitive subgroup $G$ of the symmetric group $S_n$ we can consider how many number fields have $G$ as their Galois group. Let $\widetilde{K}$ be the Galois closure of $K$ and define  \[\mathcal{F}_n(X;G): = \{K/\Q\;|\; [K:\Q]=n, \text{Gal}(\widetilde{K}/\Q) \cong G,\; \text{Disc}(K) \leq X\}.\]As a direct result of the Hermite--Minkowski theorem, this set is finite. A conjecture of Malle \cite{2004Malle} predicts the asymptotic size of this set, 
\[\#\mathcal{F}_n(X;G) \sim c_GX^{a(G)}(\log{X})^{b(G)-1}, \tag{1.1} \label{malleconstant}\]
where $c_G$ is a constant dependent on $G$ and the values $a(G)$ and $b(G)$ can be explicitly computed (see also Section \ref{subsec:grouptheory}). This conjecture has been proven in several cases (for example \cite{2005Bhargava, 2010Bhargava, 1971DavHeil, 1989Wright}), but remains out of reach in general. A weaker version of the conjecture \cite{2000Malle} predicts lower and upper bounds for the size of this set: for all $\varepsilon >0$,
\[X^{a(G)} \ll \#\mathcal{F}_n(X;G) \ll X^{a(G)+\varepsilon}.  \tag{1.2}\label{Malleweak}\]

The weaker form of the conjecture is similarly unresolved. For $S_n$ and the alternating group $A_n$, general lower bounds have been achieved \cite{2022Bhargava, 2021LandLOThorne}, but the exponents proven are about one-half and one-fourth of their conjectured values, respectively. Similarly, a lower bound was recently proven for $\text{GL}_2(\mathbb{F}_\ell)$ in its regular representation \cite{2023Ray}, however its exponent is about one-twelfth that of what Malle predicts; in this paper we give improved results for $\text{GL}_2(\mathbb{F}_\ell)$ in every transitive permutation representation. 

\subsection{Main Result} 

Let $\mathbb{F}_\ell$ be the finite field of $\ell$ elements. We can view $\text{GL}_2(\mathbb{F}_\ell)$ as a transitive subgroup of $S_{\ell^2-1}$ by considering its action on $\mathbb{F}_\ell^2\setminus\{0\}$. In this case, the predicted lower bound in \eqref{Malleweak} would be $X^{\frac{2}{\ell(\ell-1)}}$ (see Lemma \ref{indexcomp}). In this setting, we obtain the following lower bound. 

\begin{theorem}\label{bigboy}
Let $\ell \geq 13$ be a prime integer. Then, as $X \to \infty$,
\[\#\mathcal{F}_{\ell^2-1} (X;\rm{GL}_2(\mathbb{F}_\ell)) \gg X^{\frac{1}{2(\ell-1)^2}}.\]
\end{theorem}

The exponent in this result is about one-fourth of the one appearing in \eqref{Malleweak}, which is commensurable with the current best lower bounds for $S_n$ and $A_n$. 

From the Galois correspondence, we know that a $\text{GL}_2(\mathbb{F}_\ell)$ number field will contain a subfield whose Galois group is the projective general linear group $\text{PGL}_2(\mathbb{F}_\ell)$. By virtue of its action on $\mathbb{P}^1(\mathbb{F}_\ell)$, $\text{PGL}_2(\mathbb{F}_\ell)$ has a primitive, degree $\ell+1$ permutation representation. Moreover, a GL$_2(\mathbb{F}_\ell)$-extension of degree $\ell^2-1$ will contain a PGL$_2(\mathbb{F}_\ell)$-subextension of degree $\ell+1$. Using similar ideas from our proof of the GL$_2(\mathbb{F}_\ell)$ lower bound, we also get a PGL$_2(\mathbb{F}_\ell)$ lower bound. 

\begin{theorem}\label{smallboy}
For $\ell \geq 13$, as $X \to \infty$,
\[\#\mathcal{F}_{\ell+1} (X;\rm{PGL}_2(\mathbb{F}_\ell)) \gg X^{\frac{1}{2(\ell-1)}}.\]
\end{theorem}

As indicated above, the exponent in this lower bound is about one-fourth of the anticipated exponent appearing in \eqref{Malleweak}. (In the case of the degree $\ell+1$ permutation representation of $\text{PGL}_2(\mathbb{F}_\ell)$, the exponent is predicted to be $2/(\ell-1)$.) By contrast, the best known upper bound for $\text{PGL}_2(\mathbb{F}_\ell)$, given in Theorem 5.5 of \cite{2023LO}, has an exponent that is larger by a factor of more than $2\ell$ than that of the one in \eqref{Malleweak}. 

We will see that the method in which we go about proving these theorems can be applied to other permutation representations of both $\text{GL}_2(\mathbb{F}_\ell)$ and $\text{PGL}_2(\mathbb{F}_\ell)$. We outline this procedure and give an example as to how this work can be generalized in Section \ref{sec:generalizations}. 

\subsection{Method of Proof} \label{sec:method}
Let $E$ be a rational elliptic curve in short Weierstrass form $E:y^2=x^3+Ax+B$. For $\ell$ prime, let $E[\ell]$ denote the set of $\ell$-torsion points of $E$. The $\ell$-torsion field $\mathbb{Q}(E[\ell])$, generated by the coordinates of the $\ell$-torsion points, is a Galois extension of $\mathbb{Q}$ whose Galois group is a subgroup of $\text{GL}_2(\mathbb{F}_\ell)$. More explicitly, if we let 
\[\rho_{E,\ell} \colon \text{Gal}(\mathbb{Q}(E[\ell])/\mathbb{Q}) \to\text{GL}_2(\mathbb{F}_\ell),\tag{1.3}\label{galrep}\]
be the natural mod-$\ell$ representation of the $\ell$-torsion field, the fields we are interested in will be precisely those with $\text{im}(\rho_{E,\ell})=\text{GL}_2(\mathbb{F}_\ell)$. 
In this paper, we consider the subfields $\mathbb{Q}(P) \subseteq \mathbb{Q}(E[\ell])$ obtained from adjoining just one $\ell$-torsion point, $P=(x,y)$, to $\mathbb{Q}$. When $\rho_{E,\ell}$ is surjective, this subfield will have $\text{GL}_2(\mathbb{F}_\ell)$ as its Galois group, viewed in its degree $\ell^2-1$ permutation representation. 

We use the discriminants of elliptic curves to describe the discriminants of their $\ell$-torsion fields, as the two are closely related. Given an elliptic curve in short Weierstrass form such that $4A^3+27B^2$ is square-free and $\rho_{E,\ell}$ is surjective, we are able to control the inertia subgroup associated to $\mathbb{Q}(E[\ell])$ and hence control the discriminants of its subfields. We then use modular curves and what's known about rational $\ell-$isogenies of elliptic curves to show $\rho_{E,\ell}$ will be surjective for all but finitely many $E/\mathbb{Q}$ of our choosing for $\ell \geq 13$. From there, we use a counting argument to prove Theorems \ref{bigboy} and \ref{smallboy}.

\subsection*{Acknowledgements} This paper grew out of a Master's thesis completed in Spring 2023, under the  advisement of Robert Lemke Oliver. During this time, I was informed that Anwesh Ray had been independently working on an adjacent problem, and some related results were published by him in \cite{2023Ray}. I would like to thank Daniel Keliher, John Voight, Tristan Phillips, and Álvaro Lozano-Robledo for their helpful comments and guidance.

\section{Galois groups as permutation groups}

In this section, we review permutation representations of finite groups and describe what the Galois group of a non-Galois extension is. We explicitly define the constant $a(G)$ from Malle's conjecture (\ref{malleconstant}) and compute it for $\text{GL}_2(\mathbb{F}_\ell)$ and $\text{PGL}_2(\mathbb{F}_\ell)$. 

\subsection{Permutation Representations}\label{sec:perm}

Let $G$ be a finite group of order $n$. For each $g \in G$, the action of left multiplication by $g$ produces a permutation on the $n$ elements of the group. In this way, there is a natural isomorphism between $G$ and a subgroup of $S_n$. We consider this to be the \textit{regular} permutation representation of $G$. 

We can also consider a faithful, transitive group action of $G$ on a finite set $X$ of $d$ elements. Each $g \in G$ induces a permutation on the elements of $X$ so the action is isomorphic to a subgroup of $S_{d}$. Formally, a \textit{permutation representation} of a group $G$ acting on a set $X$ is a homomorphism $\pi \colon G \to S_{d}$. In this paper we will call the image of the respective representation a \textit{degree d permutation representation} of $G$.

For primes $\ell \geq 2$, one can consider the natural action of $\text{GL}_2(\mathbb{F}_\ell)$ on $\mathbb{F}_\ell^2\setminus\{0\}$, which gives rise to a degree $\ell^2-1$ permutation representation of $\text{GL}_2(\mathbb{F}_\ell)$. Similarly, for $\text{PGL}_2(\mathbb{F}_\ell)$, there is a natural action on the $\ell+1$ elements of $\mathbb{P}^1(\mathbb{F}_\ell)$, which gives rise to a degree $\ell+1$ permutation representation of the group. 

\begin{remark}
The majority of our discussion will be dedicated to these specific representations. In Section \ref{sec:generalizations}, we generalize to other representations of these groups.
\end{remark}
 
We can use the notion of permutation representations of a group to define Galois groups for non-Galois extensions. Let $K$ be a number field of degree $n$ and let $\widetilde{K}$ be its Galois closure over $\mathbb{Q}$. There is a natural faithful and transitive group action of $\Gal(\widetilde{K}/\Q)$ on the $n$ embeddings $K \hookrightarrow \overline{\mathbb{Q}}$. The image of the degree $n$ representation of $\Gal(\widetilde{K}/\Q)$ in $S_n$ is what we define to be the Galois group of $K/\mathbb{Q}$. 

\subsection{Malle's index} \label{subsec:grouptheory} Let $G$ be a transitive subgroup of $S_n$. Here we explicitly define the exponent $a(G)$ in Malle's conjecture. For an element $g \in G$, we define its \textit{index} by
\[\text{ind}(g)=n-\#\text{orbits $g$ induces on }\{1,\ldots,n\}. \tag{2.1}\label{ind}\]
We use this to then define the \textit{Malle index} of a group: \[\text{ind}(G) = \text{min}\l\{\text{ind}(g) \; | \; g \in G, \; g\ne\text{Id}\r\}.\] 

\begin{definition}
The constant in Malle's conjecture, $a(G)$, is defined as
\[a(G)=\rm{ind}(G)^{-1}. \tag{2.2}\label{malexp}\]
\end{definition} 

It will be useful to compute $a(G)$ for $\text{GL}_2(\mathbb{F}_\ell)$ in its $\ell^2-1$ permutation representation, as well as $\text{PGL}_2(\mathbb{F}_\ell)$ in its $\ell+1$ permutation representation. Note that to assume $g \in G$ has minimal index amongst all elements of $G$ is to assume that $g$ induces the maximal number of orbits on $n$ elements.

\begin{lemma}\label{indexcomp}
Let $\ell \geq 3$. Viewed in its degree $\ell^2-1$ permutation representation, the exponent in Malle's conjecture for $\rm{GL}_2(\mathbb{F}_\ell)$ is 
\[a(\rm{GL}_2(\mathbb{F}_\ell))=\frac{2}{\ell(\ell-1)}.\]
Similarly, by viewing $\rm{PGL}_2(\mathbb{F}_\ell)$ in its $\ell+1$ representation, the exponent in Malle's conjecture is \[a(\rm{PGL}_2(\mathbb{F}_\ell)) = \frac{2}{\ell-1}.\]
\end{lemma}

\begin{proof} If an element of $\text{GL}_2(\mathbb{F}_\ell)$ has no fixed points, then we know that every point must be contained within some orbit of size at least 2. In this case, the most we can have is $\frac{\ell^2-1}{2}$ orbits, which occurs when we have a product of disjoint transpositions. 

If a nonidentity element fixes one point, it must fix the $\ell-2$ other points on the same line, and it cannot fix any more points otherwise it would be the identity. Thus we need only consider the maximum number of orbits a permutation can produce when $\ell-1$ points are fixed. The maximum number of orbits would occur when the remaining $\ell^2-\ell$ points are permuted by $\frac{\ell^2-\ell}{2}$ transpositions, making our element of $\text{GL}_2(\mathbb{F}_\ell)$ a product of the transpositions with $\ell-1$ single-element orbits. In this case, the maximum number of orbits that can be produced is $ \frac{\ell^2+\ell-2}{2}$ orbits. Since $\frac{\ell^2+\ell-2}{2}>\frac{\ell^2-1}{2}$, an element that fixes $\ell-1$ points and pairs the remaining points would give us the minimum index across $\text{GL}_2(\mathbb{F}_\ell)\setminus\{\text{Id}\}$. Such an element exists; an example is

\[M=\begin{pmatrix} 1 & 0 \\ 0 & -1 \end{pmatrix}.\]

For $\ell \geq 3$, $\text{PGL}_2(\mathbb{F}_\ell)$ acts on an even number of elements. So we know that if a nonidentity element of $\text{PGL}_2(\mathbb{F}_\ell)$ fixes no points, then the maximum number of orbits it can produce is $\frac{\ell+1}{2}$.  Any element that fixes more than 2 points of $\mathbb{P}^1(\mathbb{F}_\ell)$ must actually fix every point and be the identity. 

 The maximal number of orbits a permutation with at least one fixed point can have is $2+\frac{\ell-1}{2}$ orbits. In fixing one point, you are left with $\ell$ others. Thus, you can pair up $\ell-1$ of the remaining points and fix one more to obtain $2+\frac{\ell-1}{2}$ orbits. An element with this action would have the minimum index in $\text{PGL}_2(\mathbb{F}_\ell)$; the projective image of $M$ from $(a)$ will do. 
\end{proof}

\begin{remark} When $\ell=2$, $\rm{GL}_2(\mathbb{F}_2)\cong\rm{PGL}_2(\mathbb{F}_2)\cong S_3$ and thus $a(G)=1$ as the group contains a transposition. 
\end{remark}

From these computations, we can give Malle's weak conjecture for these two groups in these specific permutation representations.

\begin{conjecture} Let $\ell \geq 3$ be prime. For all $\varepsilon>0$, 
\begin{align*}X^\frac{2}{\ell(\ell-1)} \ll & \; \#\mathcal{F}_{\ell^2-1}(X; \rm{GL}_2(\mathbb{F}_\ell)) \ll X^{\frac{2}{\ell(\ell-1)}+\varepsilon};\\
X^\frac{2}{(\ell-1)} \ll & \; \#\mathcal{F}_{\ell+1}(X;\rm{PGL}_2(\mathbb{F}_\ell)) \ll X^{\frac{2}{(\ell-1)}+\varepsilon}.
\end{align*}
\end{conjecture}

\section{Elliptic curves and $\ell$-torsion fields}

Let $\ell \geq 2$ and let $E$ be a rational elliptic curve given by the short Weierstrass equation $y^2=f(x)=x^3+Ax+B$. Recall the discriminant of a curve in this form is given by \[\Delta_E = -16(4A^3+27B^2).\] Let $\Delta_f=4A^3+27B^2$, corresponding to the discriminant of the right-hand cubic $f(x)$. For an elliptic curve written in this form, we can also define the \textit{height} to be $H_E=\text{max}\{4|A|^3, 27B^2\}$. 

It is known that for all primes $\ell$, the $\ell$-torsion subgroup is $E[\ell]\cong (\mathbb{Z}/\ell\mathbb{Z})^2$. Adjoining the coordinates of these points to $\Q$, we obtain the $\ell$-torsion field $\Q(E[\ell])$. This extension is Galois with Galois group a subgroup of $\text{GL}_2(\mathbb{F}_\ell)$ and hence its degree is at most $\#\text{GL}_2(\mathbb{F}_\ell)=(\ell^2-1)(\ell^2-\ell)$. When the Galois group is precisely $\text{GL}_2(\mathbb{F}_\ell)$, we can adjoin a single non-trivial $\ell$-torsion point, $P$, to obtain a subfield $K_{\ell^2-1}\vcentcolon = \mathbb{Q}(P)$ of degree $\ell^2-1$. By considering the action of the Galois group of $\mathbb{Q}(E[\ell])$ on the embeddings of this field we find $\text{Gal}(K_{\ell^2-1}/\mathbb{Q})$ is $\text{GL}_2(\mathbb{F}_\ell)$ in its degree $\ell^2-1$ permutation representation as described in Section \ref{sec:perm}. An alternative way of obtaining this type of subfield is by looking at the subgroup of $\text{GL}_2(\mathbb{F}_\ell)$ that fixes one $\ell$-torsion point.

Further, when the Galois group of $\mathbb{Q}(E[\ell])$ is $\text{GL}_2(\mathbb{F}_\ell)$, the subfield fixed by scalar matrices (the center of the group) is a degree $\ell(\ell^2-1)$ extension with Galois group $\text{PGL}_2(\mathbb{F}_\ell)$ by Galois correspondence. We can then find a subfield $K_{\ell+1}$ of degree $\ell+1$, contained within $K_{\ell^2-1}$, whose Galois group is permutation isomorphic to the degree $\ell+1$ permutation representation of $\text{PGL}_2(\mathbb{F}_\ell)$. 

\subsection{Representations of $\text{Gal}(\Q(E[\ell])/\Q)$}\label{modell}
Let $\rho_{E,\ell}$ denote the mod-$\ell$ Galois representation from \ref{galrep}. As suggested above, we are particularly interested in when $\rho_{E,\ell}$ is surjective.

It is known that if $\rho_{E,\ell}$ is \textit{not} surjective, then its image is contained in one of three types of maximal subgroups of $\text{GL}_2(\mathbb{F}_\ell)$: a Borel subgroup (those composed of upper triangular matrices), the normalizer of a split or non-split Cartan subgroup (diagonal matrices isomorphic to $(\mathbb{F}_\ell^\times)^2$ or matrices isomorphic to $\mathbb{F}^\times_{\ell^2}$, respctively), or an exceptional subgroup (those whose image in $\text{PGL}_2(\mathbb{F}_\ell)$ is $A_4, A_5$, or $S_4$). 

In \cite{1971Serre}, it is shown that for $\ell \geq 17$ the image of the mod-$\ell$ representation cannot be contained within an exceptional subgroup of $\text{GL}_2(\mathbb{F}_\ell)$. Additionally, it's known that, for elliptic curves defined over $\mathbb{Q}$, it is not possible for the image of $\rho_{E,\ell}$ to be $A_4$ or $A_5$ \cite{lmfdb}. To determine how often the mod-$\ell$ representation is contained within a Borel subgroup, the normalizer of a Cartan subgroup or a subgroup whose projective image is $S_4$, we use modular curves. We are concerned with how many rational points they have, as the rational points will correspond to $\overline{\Q}$-isomorphism classes of elliptic curves whose mod-$\ell$ representation is contained within a specific group. We will use the following notation: 
\begin{align*}
X_0(\ell) & := \text{ modular curve for a Borel subgroup}\\
X_s^+(\ell) & := \text{ modular curve for the normalizer of a split Cartan subgroup}\\
X_{ns}^+(\ell) & := \text{ modular curve for the normalizer of a non-split Cartan subgroup}\\
X_{S_4}(\ell) & := \text{ modular curve for subgroups whose projective image are $S_4$}
\end{align*}

The following lemma relies on what we know about rational points on these curves and their genus and will contribute to our ability to make a statement about how often the curves we're counting will have a non-surjective Galois representation. 

\begin{lemma}\label{isomorphismclasses}
Let $\ell \geq 17$. For all but finitely many $\overline{\Q}$-isomorphism classes of elliptic curves, the representation $\rho_{E,\ell}$ is surjective. 
\end{lemma}

\begin{proof}
The curves $X_0(\ell)$ for $\ell = 2,3,5,6,13$ have genus zero and have infinitely many rational points. For $\ell=11$ and $\ell \geq 17$, a theorem of Mazur \cite[Theorem 6]{1977Mazur} gives us that $X_0(\ell)$ has finitely many rational points. Thus for $\ell=11$ and $\ell \geq 17$ there are only finitely many isomorphism class of elliptic curves whose mod-$\ell$ representation is contained within a Borel subgroup. 

For $\ell\geq 17$, both $X_s^+(\ell)$ and $X_{ns}^+(\ell)$ have genus 2, so by Faltings' theorem they must have finitely many rational points. More precisely, Bilu, Parent, and Rebolledo show in \cite{2013Bilu} that for $\ell =11$ and $\ell \geq 17$, it is not at all possible for the image of $\rho_{E,\ell}$ to lie within the normalizer of a split Cartan subgroup. Thus for $\ell \geq 11$, there are only finitely many isomorphism classes of elliptic curves whose Galois representation of $\Q(E[\ell])$ is contained within the normalizer of a split Cartan subgroup.

Since we also know that $\text{im}(\rho_{E,\ell})$ cannot be contained within an exceptional subgroup of $\text{GL}_2(\mathbb{F}_\ell)$, we can say that for $\ell \geq 17$, there are only finitely many $\overline{\mathbb{Q}}$-isomorphism classes of elliptic curves with nonsurjective representation. 
\end{proof}

For other primes, we are able to place an upper bound on the number of elliptic curves with non-surjective Galois representation. 

\begin{lemma}\label{thirteen}
When $\ell=13$, there is a constant $c_{13}>0$ such that the number of elliptic curves with non-surjective Galois representation with naive height at most $X$ is asymptotically $c_{13}X^{\frac{1}{6}}$.
\end{lemma}

\begin{proof}
The curves $X_{ns}^+(13)$ and $X_s^+(13)$ are both known to have genus 3 \cite{lmfdb}. Similarly, Theorem 1.1 of \cite{2023BDMTV} shows $\#X_{S_4}^+(13)(\mathbb{Q})=4$. Thus, it remains to consider how often we expect the image of $\rho_{E,\ell}$ to be within a Borel subgroup of $GL_2(\mathbb{F}_\ell)$.

The curve $X_0(13)$ has infinitely any rational points, however in his thesis, Molnar \cite[Theorem 1.2.4]{molnarthesis} proves that there is a positive constant $c_{13}$ such that for any $\varepsilon>0$ the number of elliptic curves of height at most $X$ with a $13$-isogeny is $c_{13}X^{1/6} + O(X^{1/8+\varepsilon})$. This tells us we can expect there are asymptotically $c_{13}X^{\frac{1}{6}}$ elliptic curves whose mod-13 representation will not be surjective, but be contained within a Borel subgroup.
\end{proof}


\subsection{Discriminants of $\ell$-torsion fields}
Here we recall some useful connections between the discriminants of elliptic curves and the discriminants of their $\ell-$torsion fields. We use their relationship to place an upper bound on the discriminant of the subfield of $\mathbb{Q}(E[\ell])$ we are interested in. We begin by defining some notation to be used for the remainder of this paper.

\begin{definition}
Let $E$ be a rational elliptic curve in short Weierstrass form. Assume the  mod-$\ell$ representation $\rho_{E,\ell}$ is surjective and set $G_\ell:=\rm{GL}_2(\mathbb{F}_\ell)$. Let
\begin{enumerate}
\item $D_{\#G_\ell}(A,B)$ denote the discriminant of the $\ell-$torsion field $\mathbb{Q}(E[\ell])$, and
\item $D_{\ell^2-1}(A,B)$ denote the discriminant of $K_{\ell^2-1}=\mathbb{Q}(P)$.
\end{enumerate}
\end{definition}

From the Nerón-Ogg-Shafarevich Criterion, we know that a prime $p \ne \ell$ divides $D_{\#G_\ell}(A,B)$ if and only if $E$ has bad reduction at that prime \cite{2009Silverman}. Further, recall that the primes of bad reduction of an elliptic curve are precisely those that divide $\Delta_E$.  Thus we can expect the primes dividing $\Delta_E$ to ramify in $\mathbb{Q}(E[\ell])$ and appear in the factorization of $D_{\#G_\ell}(A,B)$. 

We focus on elliptic curves such that the discriminant $\Delta_f=4A^3+27B^2$ of the defining polynomial is square-free. 

\begin{lemma}\label{sqrfree}
Let $E/\mathbb{Q}$ be an elliptic curve given by $y^2=x^3+Ax+B$. If $\Delta_f$ is square-free, then the following are true:
\begin{enumerate}
\item The short Weierstrass form above is minimal.
\item Let $\nu_p$ denote the $p-$adic valuation associated to a prime $p$. For all primes $p$ dividing $\Delta_f$, except possibly $p=2,3$, we have $\nu_p(j(E))=-1$. 
\end{enumerate}
\end{lemma}

\begin{proof}
(1) Since the discriminant of $E$, $\Delta_E=-16\cdot \Delta_f$, is not divisible by the 12th power of any prime, we have a minimal model \cite{2009Silverman}. 

(2) In short Weierstrass form, the j-invariant is defined by $j(E)=-2^63^3\frac{4A^3}{\Delta_f}$. Since $\Delta_f$ is square free, it does not share any common prime factors with $4A^3$ except possibly $p=2,3$. Therefore, if $p|\Delta_f$, we see $\nu_p(j(E))=-1$. 
\end{proof}

To determine the exact power to which a prime divides $D_{\#G_\ell}(A,B)$, we turn to consider the inertia subgroups over ramified primes. Let $\mathcal{E}_p$ denote the inertia subgroup of the Galois group over $p$. When the representation is surjective, the power to which a tamely ramified prime will divide $D_{\#G_\ell}(A,B)$ is $p^k$ where \[k=\#\text{GL}_2(\mathbb{F}_\ell) - [\text{GL}_2(\mathbb{F}_\ell):\mathcal{E}_p]. \tag{3.1}\label{power}\]

In the Galois extension $\mathbb{Q}(E[\ell])$, the size of the inertia group is equivalent to the ramification index, $e$, at any ramified prime. The theory of ramification tells us that $e$ must divide the degree of the extension. Thus, in considering only the $\ell-$torsion fields of full degree, we see $e|(\ell^2-1)(\ell^2-\ell)$. 

From Serre's work on Galois representations \cite{1971Serre}, we have that the inertia group over a ramified prime $p\ne \ell$ is either trivial or cyclic of order $\ell$. It is known from the theory of Tate curves that $\mathcal{E}_p$ will be of order $\ell$ specifically when $\ell \nmid \nu_p(j(E))$. It follows from this and Lemma \ref{sqrfree}, that if we restrict to elliptic curves with $\Delta_f$ square-free, we will know that over primes $p|\Delta_f$ except for possibly $p=2,3$, the inertia group $\mathcal{E}_p$ is guaranteed to be of order $\ell$.

The generator of $\mathcal{E}_p$ can then be used to calculate the precise power to which a tamely ramified prime divides $D_{\ell^2-1}(A,B)$. Let $M_\mathcal{E}$ denote the image of the generator of the inertia subgroup under the Galois representation $\rho_{E,\ell}$. Suppose we are considering a subfield of $\mathbb{Q}(E[\ell])$ corresponding to a degree $d$ permutation representation $\pi$ of $\text{GL}_2(\mathbb{F}_\ell)$. Then if $p$ is tamely ramified, $\nu_p(D_{d}(A,B))=k$ where
\[k=d-\#\text{orbits $\pi(M_\mathcal{E})$ induces on \{1,\ldots,d\}.} \tag{3.2}\label{tamely}\]
Hence we will find it useful to precisely determine what $M_\mathcal{E}$ is.

\begin{lemma}\label{ellgroup}
Up to conjugation, $\rm{GL}_2(\mathbb{F}_\ell)$ has one subgroup of order $\ell$. 
\end{lemma}

\begin{proof}
The order of $\text{GL}_2(\mathbb{F}_\ell)$ is $\ell(\ell-1)^2(\ell+1)$. Since $\ell$ does not divide $\ell-1$ nor $\ell+1$, the largest power of $\ell$ dividing the order of the group is simply $\ell$. This tells us that any subgroup of this size is an $\ell-$Sylow subgroup of $\text{GL}_2(\mathbb{F}_\ell)$. Further, we know all $\ell-$Sylow subgroups will be conjugate. 
\end{proof}

We choose the subgroup of order $\ell$ generated by the transvection \[M_\mathcal{E} = \left(\begin{matrix} 1 & 1 \\ 0 & 1\end{matrix}\right) \tag{3.3}\] as our representative for $\mathcal{E}_p$ over any tamely ramified prime.

Now recall that if a prime $p$ is wildly ramified in a number field $K$, then $p$ divides the order of $\text{Gal}(K/\mathbb{Q})$. For $\ell \geq 2$, there are only finitely many primes that divide $\#\text{GL}_2(\mathbb{F}_\ell)$, so we can say there are only finitely many wildly ramified primes in $\mathbb{Q}(E[\ell])$, and thus in $K_{\ell^2-1}$.


\begin{lemma}\label{discriminant}
Let $E/\mathbb{Q}$ be an elliptic curve such that $\Delta_f:=4A^3+27B^2$ is square-free. For all $\ell  \geq 2$, there exists a constant $c_\ell>0$ such that
\[|D_{\ell^2-1}(A,B)| \leq c_\ell(\Delta_f)^{(\ell-1)^2}.\]
\end{lemma}

\begin{proof}
Let $\ell \geq 2$ and $E$ be a rational elliptic curve in short Weierstrass form $y^2=x^3+Ax+B$. We know that only the primes dividing $\Delta_E=-2^4(4A^3+27B^2)$ may divide $D_{\ell^2-1}(A,B)$. We also know $2$ will be wildly ramified, so without loss of generality we can narrow our focus to the primes dividing $\Delta_f=4A^3+27B^2$. Using the fact that most primes will be tamely ramified in our field, we can assume that for most pairs $(A,B)$, the primes dividing $\Delta_f$ will be tamely ramified. 

Now $M_\mathcal{E}$ induces $2\ell-2$ orbits on $\mathbb{F}_\ell^2\setminus\{0\}$. Using the formula in \eqref{tamely}, we have that if $p$ is tamely ramified, then $p^k||D_{\ell^2-1}(A,B)$ with 
\[k=(\ell^2-1)-(2\ell-2)=(\ell-1)^2.\]
Thus, we can say that, aside from the wildly ramified primes, at most $(\Delta_f)^{(\ell-1)^2}$ will divide $D_{\ell^2-1}(A,B)$.

Considering our wildly ramified primes, we know the exponents to which they will divide $D_{\ell^2-1}(A,B)$ are bounded. Thus we let $c_{\ell}$ be the product of any wildly ramified primes raised to the highest possible power to which they can divide the discriminant. This constant is then an upper bound for how the wildly ramified primes divide the discriminant. Moreover, we see that the largest possible discriminant $|D_{\ell^2-1}(A,B)|$ is $c_\ell(\Delta_f)^{(\ell-1)^2}$. 
\end{proof}

The method we used in this proof to determine an upper bound on the discriminant of $D_{\ell^2-1}(A,B)$ will be used in Section \ref{sec:generalizations} to determine similar upper bounds for other subfields of $\mathbb{Q}(E[\ell])$. 


\section{Proof of Results}\label{PROOF}

We are now ready to prove Theorems \ref{bigboy} and \ref{smallboy}. First, we show that by choosing elliptic curves with a fixed $A \ne 0$ and $\Delta_f$ square-free, we can guarantee that the fields we are counting are distinct. We then use our work in Section \ref{modell} to make a statement about how often the curves we choose will have surjective mod-$\ell$ representation. In Section \ref{sec:generalizations}, we generalize the method to other permutation representations of $\text{GL}_2(\mathbb{F}_\ell)$ and $\text{PGL}_2(\mathbb{F}_\ell)$. 

\subsection{Ensuring distinct fields} In our discussion before Lemma \ref{sqrfree}, we determined that for an elliptic curve $E:y^2=x^3+Ax+B$, the primes that ramify in the $\ell-$torsion field will be those that divide $\Delta_f$. Fixing an $A \ne 0$, our problem reduces to determining if, in general, different choices of $B$ will produce discriminants with different primes in their factorization.

\begin{lemma}\label{squarefree}
The discriminant $\Delta_f = 4A^3+27B^2$ will be square-free for a positive proportion of $B \in \mathbb{N}$. 
\end{lemma}

\begin{proof}
Fix $A \ne 0$. Then the discriminant, $\Delta_f=27B^2+4A^3$, is a quadratic polynomial in the ring $\mathbb{Z}[B]$. A result of Ricci \cite[Theorem A]{1933Ricci} gives us that for a positive proportion of $B$, $\Delta_f(B)$ will be square-free. 
\end{proof}

Combined with our analysis of the valuation of $D_{\ell^2-1}(A,B)$ at tamely ramified primes as previously discussed, this lemma allows us to prove the following.

\begin{lemma}\label{FIELDS}
Fix $A \ne 0$. Let $E:y^2=f(x)=x^3+Ax+B$ and $E':y^2=g(x)=x^3+Ax+B'$ be rational elliptic curves with surjective mod-$\ell$ representation and $B \ne B'$ and let $P_E$ and $P_{E'}$ be $\ell$-torsion points on each curve, respectively. If $\Delta_f \ne \Delta_g$ are square-free, then $\mathbb{Q}(P_E)$ and $\mathbb{Q}(P_{E'})$ are distinct fields.  
\end{lemma}

\begin{proof}
Choose $B \ne B' \in \mathbb{N}$ such that the elliptic curves $E:y^2=f(x) = x^3+Ax+B$ and $E':y^2=g(x)=x^3+Ax+B'$ have surjective Galois representation and such that $\Delta_f$ and $\Delta_g$ are square-free. If $4A^3+27B^2$ and $4A^3+27(B')^2$ are both square-free and $B \ne B'$, then their prime factorizations must not coincide. Therefore, the prime factorizations of $D_{\ell^2-1}(A,B)$ and $D_{\ell^2-1}(A,B')$ will not coincide, giving us that $\Q(P_E)$ and $\Q(P_{E'})$ are distinct fields.
\end{proof}

By fixing $A$ and varying $B \in \mathbb{N}$, we are now able to guarantee that each elliptic curve we choose will produce a different $\ell$-torsion field and, hence, different subfields. 

\subsection{Surjectivity of $\rho_{E,\ell}$}\label{sec:surj}

From Lemma \ref{isomorphismclasses} we know that for $\ell \geq 17$ and for all but finitely many $\overline{\mathbb{Q}}$-isomorphism classes, the representation $\rho_{E,\ell}$ is surjective. Similarly, from Lemma \ref{thirteen} we have an asymptotic size of the number of curves with non-surjective mod-13 representation. For rational elliptic curves  with $j-$invariant not equal to 0 or 1728, $\overline{\mathbb{Q}}$-isomorphism classes correspond to quadratic twists by square-free integers of elliptic curves. Continuing our method of fixing the $A$ coefficient on our curve, we can use this to make a stronger statement regarding surjectivity for different choices of $B$. 

\begin{lemma}\label{lemma1}
Fix $A \ne 0$. For $\ell \geq 13$ and all but finitely many $B$, the representation $\rho_{E,\ell}$ is surjective. 
\end{lemma}

\begin{proof}
Fix $A \ne 0$ and let $E(A,B)$ denote the elliptic curve given by $y^2=x^3+Ax+B$. We start with the simpler case of $\ell \geq 17$. From Lemma \ref{isomorphismclasses}, we know that there are only finitely many $\overline{\Q}-$isomorphism classes of elliptic curves $E$ such that $\rho_{E,\ell}$ is not surjective. Since we have fixed $A\ne 0$, the only quadratic twists of the elliptic curve $E(A,B)$ with the same value of $A$ for any $B \in \mathbb{N}$ are those given by $y^2=x^3+Ax+B$ and $y^2=x^3+Ax-B$. Thus each $\overline{\Q}$-isomorphism class contains at most two elliptic curves with our fixed $A$. Pulling this together, we obtain that for a fixed $A$, there are only finitely many choices of $B$ for which the representation $\rho_{E,\ell}$ is not surjective.

Now consider the specific case of $\ell=13$. From Lemma \ref{thirteen}, we know there are $\sim c_{13}X^{\frac{1}{6}}$ elliptic curves $E(A,B)$ such that the image of $\rho_{E,13}$ is not all of $GL_2(\mathbb{F}_\ell)$. In proving his result, Molnar uses the parameterization of elliptic curves admitting a rational 13-isogeny given by $E:y^2=x^3+f_{13}(t)x+g_{13}(t)$ where $t \in \mathbb{Q}$ and $f_{13}, g_{13} \in \mathbb{Q}[t]$ \cite[Lemma 3.2.1]{molnarthesis}. From this, we can say that if there is an elliptic curve with a rational 13-isogeny with a fixed $A \ne 0$, there must be some $t \in \mathbb{Q}$ and $\delta \in \mathbb{Q}^\times$ such that
\[A\delta^2=f_{13}(t) = -3(t^2+t+7)(t^2+4)(t^4-235t^3+1211t^2-1660t+6256).\tag{4.1}\label{hypA}\]The equation in \eqref{hypA} is a hyperelliptic curve of genus 3 and, by Faltings' Theorem, it has finitely many rational points. Thus, there are only finitely many pairs $(t, \delta)$ that could produce $A$ as a coefficient on an elliptic curve with a rational 13-isogeny. These finitely many pairs then produce finitely many $B=g_{13}(t)/\delta^3$ such that $E(A,B)$ has non-surjective mod-13 representation. Thus, for a fixed $A\ne 0$, there are finitely many possible $B$ such that $\rho_{E,13}$ is not surjective.
\end{proof}

\subsection{Proof of Theorem \ref{bigboy}}

\begin{proof}
Let $\ell \geq 13$ and fix a non-zero $A \in \Z$. From Lemma \ref{lemma1}, we know for all but finitely many choices of $B$, the representation associated to $E(A,B)$, $\rho_{E,\ell}$, will be surjective. Those with surjective representation will produce number fields with Galois group $\rm{GL}_2(\F_\ell)$.  From Lemma \ref{squarefree} we have that for a positive proportion of $B \in \mathbb{N}$, $\Delta_f = 4A^3+27B^2$ will be square-free and from Lemma \ref{discriminant} we know $|D_{\ell^2-1}(A,B)| \ll (\Delta_f)^{(\ell-1)^2}$ when $\Delta_f$ is square-free. For different choices of $B$, and hence for different elliptic curves meeting the aforementioned criteria, Lemma \ref{FIELDS} tells us we have different fields.

Now let us vary $B \in \mathbb{N}$ such that $B \ll X^\frac{1}{2(\ell-1)^2}$ and $\Delta_f$ is square-free so that $D_{\ell^2-1}(A,B) \leq X$. This gives us $\gg X^\frac{1}{2(\ell-1)^2}$ possible elliptic curves whose $\ell-$torsion field has Galois group $\text{GL}_2(\F_\ell)$. Thus we obtain
\[\#\mathcal{F}_{\ell^2-1}(X;\text{GL}_2(\mathbb{F}_\ell)) \gg X^\frac{1}{2(\ell-1)^2},\]
as claimed.
\end{proof}

\subsection{Proof of Theorem \ref{smallboy}}

Here we prove our lower bound for $\text{PGL}_2(\mathbb{F}_\ell)$. Much of what is needed follows directly from our proof of Theorem \ref{bigboy}; surjectivity of the Galois representation and the distinctness of the fields will come for free. It remains to consider the upper bound on the discriminant of the fields of degree $\ell+1$. Let $D_{\ell+1}(A,B)$ denote the discriminant of these fields.

\begin{lemma}\label{pgldisc}
For all $\ell \geq 2$, there exists a constant $c_\ell>0$ such that
\[|D_{\ell+1}(A,B)| \leq c_\ell(\Delta_f)^{\ell-1}.\]
\end{lemma}

The proof of this lemma follows similarly to that of Lemma \ref{discriminant} in our analysis of what primes will be tamely ramified versus wildly ramified. The missing piece is simply the action of $M_\mathcal{E}$ on $\mathbb{P}^1(\mathbb{F}_\ell)$. We note that the image of $\langle M_\mathcal{E} \rangle$ in $\text{PGL}_2(\mathbb{F}_\ell)$ remains a subgroup of order $\ell$. 

\begin{proof}
The generator $M_\mathcal{E}$ induces an $\ell-$cycle on the subset of $\mathbb{P}^1(\mathbb{F}_\ell)$ generated by $[1:1]$ and fixes the single element $[1:0]$, hence induces 2 orbits. Using the formula in \ref{power}, we have that for a tamely ramified prime, $p$, $p^k||D_{\ell+1}(A,B)$ with $k=\ell-1.$ Treating the case of wildly ramified primes as in the proof of Lemma \ref{discriminant}, we obtain our result.
\end{proof}

\noindent With this, we are ready to prove Theorem \ref{smallboy}. 

\begin{proof}
Let $\ell \geq 13$ and fix $A \ne 0$. Lemma \ref{lemma1} tells us for all but finitely many choices of $B$ $\rho_{E,\ell}$ will be surjective. From Lemma \ref{pgldisc} we know $|D_{\ell+1}(A,B)| \ll (\Delta_f)^{(\ell-1)}$ and from Lemma \ref{squarefree} we have that a positive portion of the time, $\Delta_f = 4A^3+27B^2$ will be square-free. 

As we did for $\text{GL}_2(\mathbb{F}_\ell)$, let us vary $B$ such that $|B| \ll X^\frac{1}{2(\ell-1)}$, $\Delta_f$ is square-free. This gives us $\gg X^\frac{1}{2(\ell-1)}$ possible elliptic curves whose $\ell-$torsion field has Galois group $\text{GL}_2(\F_\ell)$. In turn, our subfield of degree $\ell+1$ will have Galois group $\text{PGL}_2(\mathbb{F}_\ell)$. Thus we obtain
\[\#\mathcal{F}_{\ell+1}(X;\text{PGL}_2(\mathbb{F}_\ell)) \gg X^\frac{1}{2(\ell-1)},\]
as claimed.
\end{proof}

\subsection{Generalization to other representations of $\rm{GL}_2(\mathbb{F}_\ell)$ and $\rm{PGL}_2(\mathbb{F}_\ell)$} \label{sec:generalizations}

The procedure developed in order to prove Theorems \ref{bigboy} and \ref{smallboy} can be generalized to any other permutation representation of $\text{GL}_2(\mathbb{F}_\ell)$ or $\text{PGL}_2(\mathbb{F}_\ell)$. As before, by specifying to elliptic curves with surjective mod-$\ell$ representation and square-free discriminant, we guarantee the existence and distinctness of a field corresponding to any degree permutation representation of $\text{GL}_2(\mathbb{F}_\ell)$ or $\text{PGL}_2(\mathbb{F}_\ell)$.

\begin{lemma}\label{gendisc}
Let $\pi$ be a faithful and transitive permutation representation of $\rm{GL}_2(\mathbb{F}_\ell)$ of degree d and let $D_d(A,B)$ denote the discriminant of the subfield of $\mathbb{Q}(E[\ell])$ corresponding to the representation. For all $\ell \geq 2$, there exists a constant $c_\ell>0$ such that
\[|D_{d}(A,B)| \leq c_\ell(\Delta_f)^{\text{\rm ind}(\pi(M_\mathcal{E}))},\]
where $\text{\rm ind}(\pi(M_\mathcal{E}))$ is the Malle index of $\pi(M_\mathcal{E})$. 
\end{lemma}

The proof of this theorem follows similar to the proofs of Lemma \ref{discriminant} and \ref{pgldisc}. A similar statement can be made for the case of $\rm{PGL}_2(\mathbb{F}_\ell)$. With this, we can prove the following generalizations of Theorems \ref{bigboy} and \ref{smallboy}. 

\begin{theorem}\label{almostcorollary}
Let $\pi$ be a faithful and transitive permutation representation of $\rm{GL}_2(\mathbb{F}_\ell)$ of degree d. Then for $\ell \geq 13$,
\[\#\mathcal{F}_d(X;\pi(\rm{GL}_2(\mathbb{F}_\ell))) \gg X^{\frac{1}{2(\text{\rm ind}(\pi(M_{\mathcal{E}})))}}.\] 
Similarly, if $\pi$ is a faithful and transitive permutation representation of $\rm{PGL}_2(\mathbb{F}_\ell)$ of degree d, then for $\ell \geq 13$,
\[\#\mathcal{F}_d(X;\pi(\rm{PGL}_2(\mathbb{F}_\ell))) \gg X^{\frac{1}{2(\text{\rm ind}(\pi(M_{\mathcal{E}})))}}.\] 
\end{theorem}

\begin{proof}
Let $E/\Q$ be a rational elliptic curve in short Weierstrass form with surjective representation and let $\ell \geq 13$. 

From Lemma \ref{gendisc} we know $|D_{d}(A,B)| \ll (\Delta_f)^{\text{ind}(\pi(M_\mathcal{E}))}$ and from Lemma \ref{squarefree} we have that a positive proportion of the time, $\Delta_f = 4A^3+27B^2$ will be square-free. 

Now let us vary $B$ such that $|B| \ll X^\frac{1}{2(\text{ind}(\pi(M_\mathcal{E})))}$ and $\Delta_f$ is square-free. This gives us $\gg X^\frac{1}{2(\text{ind}(\pi(M_\mathcal{E})))}$ possible elliptic curves whose $\ell-$torsion field has Galois group $\text{GL}_2(\F_\ell)$. Thus we obtain
\[\#\mathcal{F}_d(X;\pi(\text{GL}_2(\mathbb{F}_\ell))) \gg X^\frac{1}{2(\text{ind}(\pi(M_\mathcal{E})))}.\]
The proof for the degree $d$ permutation representation of $\text{PGL}_2(\mathbb{F}_\ell)$ follows similarly.
\end{proof}

\begin{example} \rm
We can use Theorem \ref{almostcorollary} to give a lower bound on the number of $\rm{GL}_2(\mathbb{F}_\ell)$ of degree $(\ell^2-1)(\ell^2-\ell)$.

Since we are considering $\text{GL}_2(\mathbb{F}_\ell)$ in its regular permutation representation, we simply have \[\text{ind}(M_\mathcal{E})=\#\text{GL}_2(\mathbb{F}_\ell)-[\text{GL}_2(\mathbb{F}_\ell):\langle M_\mathcal{E}\rangle]=(\ell+1)(\ell-1)^3.\tag{4.1}\label{regrep}\]
Thus we can say,
\[\#\mathcal{F}_{\#\text{GL}_2(\mathbb{F}_\ell)}(X;\text{GL}_2(\mathbb{F}_\ell)) \gg X^\frac{1}{2(\ell+1)(\ell-1)^3}.\]
As mentioned earlier, this result improves the previously best known lower bound of \cite{2023Ray} for the regular permutation representation of $\text{GL}_2(\mathbb{F}_\ell)$,
\[\#\mathcal{F}_{\#\text{GL}_2(\mathbb{F}_\ell)}(X;\text{GL}_2(\mathbb{F}_\ell)) \gg \frac{X^{\frac{1}{12(\ell+1)(\ell-1)^3}}}{\log{X}}.\]
In particular, the exponent in \ref{regrep} is about one-fourth that of what is predicted, whereas the one appearing in Ray's lower bound is about one-twelfth.
\end{example}

{\singlespacing
\addcontentsline{toc}{chapter}{References}
\bibliography{refs2/thesisbib}
\bibliographystyle{alpha}  
}

\end{document}